\documentclass [11pt ]{amsart} 
\usepackage {amssymb}
\usepackage{graphicx}
\usepackage{hyperref}

\newtheorem{thm}{Theorem}[subsection]

\newtheorem{cor}[thm]{Corollary}

\newtheorem{Def}[thm]{Definition}
\newtheorem{prop}[thm]{Proposition}
\newtheorem{rem}[thm]{Remark}
\newtheorem{ex}[thm]{Example}

\newcommand{\bdfn}{\begin{Def} \rm}
\newcommand{\edfn}{\end{Def}}

\newcommand{\ra}{\rightarrow}
\newcommand{\Ra}{\Rightarrow}

\newcommand{\Lglra}{\Longleftrightarrow}

\newcommand{\es}{\emptyset}
\newcommand{\ci}{\subseteq}

\newcommand{\e}{\varepsilon}

\newcommand{\la}{\lambda}

\newcommand{\mb}{\mathbb}
\newcommand{\mc}{\mathcal}

\newcommand{\sm}{\setminus}

\newcommand{\iy}{\infty}

\newcommand{\beqa}{\begin{eqnarray*}}
\newcommand{\eeqa}{\end{eqnarray*}}

\newcounter{cnt1}
\newcounter{cnt2}
\newcounter{cnt3}
\newcounter{cnt4}
\newcommand{\blr}{\begin{list}{$($\roman{cnt1}$)$} {\usecounter{cnt1}
 \setlength{\topsep}{0pt} \setlength{\itemsep}{0pt}}}
\newcommand{\blR}{\begin{list}{\Roman{cnt4}.\ } {\usecounter{cnt4}
 \setlength{\topsep}{0pt} \setlength{\itemsep}{0pt}}}
\newcommand{\bla}{\begin{list}{$(\alph{cnt2})$} {\usecounter{cnt2}
 \setlength{\topsep}{0pt} \setlength{\itemsep}{0pt}}}
\newcommand{\bln}{\begin{list}{$($\arabic{cnt3}$)$} {\usecounter{cnt3}
 \setlength{\topsep}{0pt} \setlength{\itemsep}{0pt}}}
\newcommand{\el}{\end{list}}

\begin{document}

\title[Uniqueness of Hahn--Banach extensions and its variants ]{Uniqueness of Hahn--Banach extensions and some of its variants }

\author[Daptari ]{Soumitra Daptari}
\author[Paul]{Tanmoy Paul}

\address{Dept. of Mathematics\\
Indian Institute of Technology Hyderabad\\
, India}

\email{ma17resch11003@iith.ac.in \&
tanmoy@iith.ac.in}

\subjclass[2020]{Primary 46A22, 46B10, 46B25; Secondary 46B20, 46B22. \hfill
\textbf{\today}}


\maketitle
\begin{abstract}
In this study, we analyze the various strengthening and weakening of the uniqueness of the Hahn--Banach extension. In addition, we consider the case in which $Y$ is an ideal of $X$. In this context, we study the property-$(U)/ (SU)/ (HB)$ and property-$(k-U)$ for a subspace $Y$ of a Banach space $X$. 
We obtain various new characterizations of these properties.
We discuss various examples in the classical Banach spaces, where the aforementioned properties are satisfied and where they fail. It is observed that a hyperplane in $c_0$ has property-$(HB)$ if and only if it is an $M$-summand. Considering $X, Z$ as Banach spaces and $Y$ as a subspace of $Z$, by identifying $(X\widehat{\otimes}_\pi Y)^*\cong \mc{L}(X,Y^*)$, we observe that an isometry in  $\mc{L}(X,Y^*)$ has a unique norm-preserving extension over $(X\widehat{\otimes}_\pi Z)$ if $Y$ has property-$(SU)$ in $Z$.
It is observed that a finite dimensional subspace $Y$ of $c_0$ has property-$(k-U)$ in $c_0$, and if $Y$ is an ideal, then $Y^*$ is a $k$-strictly convex subspace of $\ell_1$ for some natural $k$. 
\end{abstract}
\maketitle

{\centering\footnotesize dedicated to the memory of Professor Ashoke K. Roy\par}

\section{Introduction}
The Hahn--Banach extension theorem is one of the most long-standing topics in functional analysis. Many illustrious mathematicians including Lawrence Narici, Edward Beckenstein, A. R. Halbrook, Mart Poldvere, Dirk Werner, and Eve Oja have contributed to the understanding of this continuously-discussed theorem. The reader can follow the link in \cite{MO} to obtain an overview of the importance of this theorem. 
This paper follows the roadmap of the unique existence of norm-preserving extensions. For a subspace $Y$ of $X$, if the norm-preserving extension over $X$ is unique for every $y^*\in Y^*$, the subspace is said to have {\it property-$(U)$} (see \cite{P}). A slightly weaker property than that stated above is given as follows.

We say $Y$ has {\it property-$(wU)$} whenever $y^*\in Y^*$ satisfying $\|y^*\|=y^*(y_0)$ for some $y_0\in S_Y$ has a unique norm-preserving extension to $X$. The property-$(wU)$ is known as the {\it weak Hahn--Banach extension property} (see \cite{LA}) in the literature; we maintain similarities with this property while referring to the other properties along with it. We continue to study the uniqueness of the Hahn--Banach extension.
In a recent paper (\cite{DPR}), the property-$(SU)$ was extensively studied. A potpourri presentation of the said property led us to further strengthen  this property. In this paper, we focus on property-$(wU)$, property-$(HB)$, and property-$(k-U)$.

We now list a few notations and definitions that are used throughout this manuscript.

\subsection{Notations and Definitions} 
$X$ denotes a Banach space over real scalars. We consider $Y$ to be a closed subspace of $X$. $B_X$ and $S_X$ represent the closed unit ball and the unit sphere of $X$, respectively. 
\begin{enumerate}
	\item $Y^\#=\{x^*\in X^*:\|x^*\|=\|x^*|_Y\|\}$.
	\item $HB(y^*)=\{x^*\in X^*:\|x^*\|=\|y^*\|~\&~x^*|_Y=y^*\}$, for $y^*\in Y^*$.
	\item For $A\ci X$, $\dim A=\dim (\textrm{span} \{A-a\})$, where $a\in A$.
	\item For $x\in X\sm\{0\}$, $S(X^*,x)=\{x^*\in S_{X^*}:x^*(x)=\|x\|\}$.
	\item $Y^\perp=\{x^*\in X^*:x^*|_Y=0\}$.
	\item $d(x,Y)=\inf_{y\in Y}\|x-y\|$ and $P_Y(x)=\{y\in Y:\|x-y\|=d(x, Y)\}$.
	\item For Banach spaces $X, Z$, $X\cong Z$ means that $X$ is isometrically isomorphic to $Z$.
\end{enumerate}

We now provide a few important definitions that are relevant to the central theme of this investigation.

A projection $P:X\ra X$ is a linear map with $\|P\|<\iy$, unless otherwise stated, such that $P^2=P$. By a {\it contractive projection} $P$ on $X$, we mean $P:X\ra X$ such that $\|P\|=1$. A {\it bi-contractive projection} $P:X\ra X$ is a projection where $\|P\|=1=\|I-P\|$.

\bdfn
A subspace $Y$ is called an {\it ideal} in $X$ if there exists a contractive projection $P$ on $X^*$ with $\ker (P)=Y^\perp$.
\edfn

A few stronger versions of property-$(U)$ are the following.

\bdfn\label{D1}
Let $Y$ be a subspace of $X$, and suppose there exists a projection $P:X^*\ra Y^\perp$ with $\|P\|=1$. Let $G=(I-P)(X^*)$.
\bla
\item \cite{O}  $Y$ is said to have {\it property-$(SU)$} in $X$ if for each $x^*\in X^*$,
$$\|x^*\|>\|y^\#\|,$$
$\text{ whenever } x^*=y^\#+y^\perp \text{ with } y^\#\in G, y^\perp(\neq 0) \in Y^{\perp}.$
\item \cite{H} $Y$ is said to have {\it property-$(HB)$} in $X$ if for each $x^*\in X^*$,
$$\|x^*\|>\|y^\#\| ~\&~ \|x^*\|\geq \|y^\perp\|,$$
$$\text{ whenever } x^*=y^\#+y^\perp \text{ with } y^\#\in G, y^\perp(\neq 0) \in Y^{\perp}.$$
\el
\edfn

\bdfn \cite{HWW}
A subspace $J$ is called an {\it $M$-ideal} in $X$ if there exists a subspace $Z$ of $X^*$ such that $X^*=J^{\perp}\oplus_{\ell_1} Z$.
\edfn

Every $M$-ideal has property-$(HB)$. 
A subspace $J$ is said to be {\it semi $M$-ideal} if $X^*=J^\perp\oplus_{\ell_1} Z$ for some closed set $Z$ in $X^*$. The obvious choice for $Z$ for the (semi) $M$-ideal is $J^\#$. A semi $M$-ideal with a linear $J^\#$ is an $M$-ideal. $J$ is a (semi) $M$-ideal if and only if it satisfies the $(2) 3$-ball property in $X$ (see \cite[Theorem 6.9, 6.10]{L}).

From \cite{LH} and \cite{HWW}, we recall that a Banach space $X$ is called {\it $L_1$-predual} if $X^*$ is isometrically isomorphic to $L_1(\mu)$ for some measure space $(\Omega,\Sigma,\mu)$, and $X$ is called {\it $M$-embedded} if its canonical image in $X^{**}$ is an $M$-ideal in $X^{**}$. 

By a {\it hyperplane} in $X$, we mean a subspace of type $\{x\in X:x^*(x)=0\}$ for some $x^*\in X^*$. A finite co-dimensional subspace is a finite intersection of the hyperplanes in $X$.

\bdfn \label{D2}\cite{X}
A subspace $Y$ is said to have {\it property-$(k-U)$} in $X$, for some $k\in\mb{N}$ if $\dim HB(y^*)\leq k-1$ for all $y^*\in S_{Y^*}$.
\edfn

Let us recall that $Y$ is said to be a {\it proximinal} subspace of $X$ if $P_Y(x)\neq\es$ for all $x$. Subspace $Y$ is said to be a {\it Chebyshev subspace} if $P_Y(x)$ is a singleton for all $x\in X$. In general, the map $P_Y:X\ra 2^Y$ is closed convex valued and not necessarily linear if $Y$ is Chebyshev. We call $P_Y$ the {\it metric projection} of $Y$.

\bdfn\cite{S}\label{D3}
Let $k\in \mathbb{N}$.
\bla 
\item A proximinal subspace $Y$ of $X$ is said to be a {\it $k$-Chebyshev} subspace if $0\leq \dim (P_Y(x))\leq k-1$ for all $x\in X$.
\item $X$ is said to be {\it $k$-strictly convex} if for any $k+1$ linearly independent elements $x_1, x_2,\dots, x_{k+1}\in S_X$, we have $\|\sum_{i=1}^{k+1}x_i\|<k+1$.
\el
\edfn 
 For a closed and bounded set $B$ of a normed space $X$, by $\textrm{ext} (B)$, we denote the set of all {\it extreme points} of $B$.
 
\subsection{A short background of this work} The smoothness of a (non-zero) vector $x$ in a Banach space is related to the uniqueness of the functional in its state space $S(X^*,x)$. Property-$(U)$ (and property-$(wU)$) extends this notion to subspaces. If there exists a linear Hahn--Banach extension operator $S:Y^*\ra X^*$ (i.e., $Y$ is an ideal in $X$), property-$(U)$ ensures the uniqueness of the existence of such an $S$ and also establishes an embedding of $Y^*$ into $X^*$; in this case, $X^*=Y^\#\bigoplus Y^\perp$, a vector space decomposition. Note that property-$(HB)$ ensures that the corresponding projection is bi-contractive. It is relevant that property-$(U)$ is not invariant under isometry. 

 We also extend our study to cases in which a subspace $Y$ fails to have property-$(U)$, but the dimension of the set $HB(y^*)$ has a finite upper bound for $y^*\in Y^*$, known as property-$(k-U)$, as defined in Definition~\ref{D2}. If $Y$ is a subspace of $X$, then
 
 $Y$ has property-$(U)$ in $X$ $\Lglra$ $Y^\perp$ is a Chebyshev subspace of $X^*$ \cite{P}.  
 
 $X^*$ is a strictly convex $\Lglra$ every subspace of $X$ has property-$(U)$ \cite{T, F}.
 
 The $k$-valued analogues of these cases are as follows.
 
 \begin{thm}\label{T11}\cite{X}
 	Let $k\in \mathbb{N}$ and $Y$ be a subspace of $X$.
 	\bla
 	\item $Y$ has property-$(k-U)$ in $X$ if and only if $Y^\perp$ is a $k$-Chebyshev subspace of $X^*$.
 	\item  $X^*$ is $k$-strictly convex if and only if every subspace of $X$ has property-$(k-U)$ in $X$.
 	\el 
 \end{thm}
 
  In this case, (property-$(k-U)$), the corresponding notion of smoothness is multismooth, as discussed in \cite{LR}.
 It is clear from the definition that an $n$ co-dimensional subspace $Y$ has property-$((n+1)-U)$; consequently, the James space has property-$(k-U)$ in its bi-dual. By contrast, no subspace of $\ell_\infty/c_0$ (and similarly for $B(\ell_2)/K(\ell_2)$) has property-$(k-U)$ in their respective superspaces; {\it a fortiori}, if $u$ is unitary in a $C^*$-algebra $\mc{A}$, then $\textrm{span}(S(\mc{A}^*,u))=\mc{A}^*$.

\subsection{Major outcomes of this work } We organize this work in the following manner.

In Section~$2$, we discuss property-$(wU)$ in Banach spaces. In Theorem~ \ref{T16}, we provide new characterizations of property-$(wU)$ in Banach space. The $3$-space problem for property-$(wU)$ is partially proven in this section.

Section~3 discusses the stability properties of various types of uniqueness of the Hahn--Banach extensions stated earlier, namely property-$(SU), (HB)$. Some new characterizations are given for these properties along with property-$(U)$ in Banach spaces in Theorems~\ref{T4}, \ref{T6}, and \ref{T1}. We partially solve the 3-space problem for property-$(HB)$.
Some examples are given at the end, establishing that property-$(wU)$($(SU)$) is strictly weaker than property-$(U)$($(HB)$).

In Section~4, we discuss property-$(k-U)$ in Banach spaces. In Theorem~\ref{T5}, it is characterized when a reflexive subspace $Y$ has property-$(k-U)$ in a Banach space $X$. In Theorem~\ref{T14}, we discuss a class of finite-dimensional subspaces of $L_1(\mu)$ that does not have property-$(k-U)$ for any $k$.
Some examples are also given with this property. 

\section{On weakly Hahn--Banach smoothness}
In this section, we study property-$(wU)$. We begin with some characterizations of property-$(wU)$. Let us recall that for $x^*\in X^*$, a hyperplane $(x^*)^{-1}\left(0\right)$ is proximinal if and only if $x^*$ attains its norm in the unit sphere of $X$.

\subsection{Characterizations}
\begin{thm}\label{T16}
	Let $Y$ be a subspace of the Banach space $X$. Then, the following are equivalent.
	\bla
	\item $Y$ has property-$(wU)$ in $X$.
	\item Every proximinal subspace $Z\ci Y$, $Y/Z$ has property-$(wU)$ in $X/Z$.
	\item $P_{Y^\perp}(x^*)=\{0\}$ for $x^*\in X^*$ satisfying $x^*(y)=\|x^*\|$, for some $y\in S_Y$.
	\el
\end{thm}

\begin{proof}
	$(a)\Ra (b)$. Suppose $Y$ has property-$(wU)$ in $X$, and $Z$ is proximinal in $Y$. Let $F\in (Y/Z)^*$ be a norm attaining functional. That is, there exists $y+Z\in Y/Z, \|y+Z\|=1$ such that $F(y+Z)=\|F\|$. Considering $F$ as an element in $Z^\perp_{Y^*}$, $y^*$, say, is clearly norm-attaining at some $y+z\in S_X$, for some $z\in Z$. Now, $y^*$ has a unique norm-preserving extension to $X$; let us call it $x^*$. Since $x^*\in Z^\perp_{X^*}$, we may identify $x^*$ to some $\widetilde{F}\in (X/Z)^*$ with $\|F\|=\|y^*\|=\|x^*\|=\|\widetilde{F}\|$. To prove the uniqueness of $\widetilde{F}$, let us assume $G\in (X/Z)^*$ such that $G|_{Y/Z}=F$. Clearly, $G\in Z^\perp_{X^*}$, which we call $w^*$. Then, $w^*|_Y=y^*$ and $\|w^*\|=\|y^*\|$. Hence, we have $w^*=x^*$, which proves the necessity of the condition.
	
	$(b)\Ra (a)$. Let us assume that $Y$ does not have property-$(wU)$ in $X$. Then, there exists $y^*\in S_{Y^*}$ such that $y^*$ attains its norm on $B_Y$ and distinct $x_1^*, x_2^*\in HB(y^*)$. Let $Z=\ker (y^*)$; then, $Z$ is proximinal in $Y$. Suppose $y_0\in B_Y$ is such that $y^*(y_0)=\|y^*\|$, then $\|y_0+Z\|=d(y_0,Z)=|y^*(y_0)|=1$. We can now identify $y^*, x_1^*, x_2^*$ in the corresponding quotient spaces, i.e., $Z^\perp_{Y^*}$ and $Z^\perp_{X^*}$, respectively. If we define $F= y^*$ and $G_1, G_2$ as $x_1^*, x_2^*$, respectively, then clearly it is a contradiction to the assumption that $Y/Z$ has property-$(wU)$ in $X/Z$. 
	
	$(a)\Ra (c).$ Let $0\neq y^\perp\in P_{Y^\perp}(x^*)$ for some $x^*\in X^*$, as stated in $(c)$. Then, $x^*, x^*-y^\perp\in HB(x^*|_Y)$. This contradicts the fact that $Y$ has property-$(wU)$ in $X$.
	
	$(c)\Ra (a).$ If $y^*\in S_{Y^*}$ such that $|y^*(y)|=1$, for some $y\in S_Y$, then, any $x^*\in HB(y^*)$, we have that $x^*(y)=1$. Hence, for distinct $x_1^*, x_2^*\in HB(y^*)$, we obtain $x_1^*-x_2^*\in P_{Y^\perp}(x_1^*)$, which contradicts our assumption on $(c)$.
\end{proof}

\subsection{A variant of the 3-space problem} We now focus on the main result of this section. Let us recall the characterization \cite[Theorem~2.2]{LA} for property-$(wU)$ of a subspace $Y$ in $X$.
\begin{thm} \label{TP3}
	Let $Y, Z$ be the subspaces of $X$ such that $Z\ci Y\ci X$. If $Z$ is a semi-M-ideal in $X$ and $Y/Z$ has property-$(wU)$ in $X/Z$, then $Y$ has property-$(wU)$ in $X$. 
\end{thm}
\begin{proof}
	Let us choose $x_0\in S_X$, $y_0\in S_Y$ and $\e>0$. 
	
	{\sc Claim:~} $\exists~R\geq 1$ such that $Y\cap B(x_0+R y_0, R+\e)\cap B(x_0-R y_0, R+\e)\neq\es.$
	
	Now, if $x_0\in Z$, then the aforementioned intersection holds true for any $R>0$ and clearly contains $x_0$. Moreover, if $y_0\in Z$, then the aforementioned intersection is non-empty for $R=1$ (see \cite[Pg.98]{LA}). Hence, without loss of generality, we assume that $x_0, y_0\notin Z$.
	
	Define $p=\frac{1}{\|x_0+Z\|}$, $q=\frac{1}{\|y_0+Z\|}$ and choose $\eta>0$ such that $\eta<\frac{\e p}{q+2p}$. 
	
	From \cite[Theorem 2.2]{LA},  there exists $r\geq 1$ such that,
	\[
	Y/Z\cap B(px_0+Z+rqy_0+Z,r+\eta)\cap B(px_0+Z-rqy_0+Z,r+\eta)\neq \phi.
	\] 
	That is,
	\[
	Y/Z\cap B(px_0+rqy_0+Z,r+\eta)\cap B(px_0-rqy_0+Z,r+\eta)\neq \phi.
	\]
	Let $r'>\max\{\frac{p}{q},r\}$. Since $r'>r$, we have 
	$$Y/Z\cap B(px_0+r'qy_0+Z,r'+\eta)\cap B(px_0-r'qy_0+Z,r'+\eta)\neq \phi.$$ 
	
	Hence, $Y/Z\cap B(x_0+\frac{r'q}{p}y_0+Z,\frac{r'}{p}+\frac{\eta}{p})\cap B(x_0-\frac{r'q}{p}y_0+Z,\frac{r'}{p}+\frac{\eta}{p})\neq \phi$. 
	
	As $q\geq 1$, we have
	
	 $Y/Z\cap B(x_0+\frac{r'q}{p}y_0+Z,\frac{qr'}{p}+\frac{q\eta}{p})\cap B(x_0-\frac{r'q}{p}y_0+Z,\frac{qr'}{p}+\frac{q\eta}{p})\neq \phi$. 
	
	Let $y+Z\in Y/Z\cap B(x_0+\frac{r'q}{p}y_0+Z,\frac{qr'}{p}+\frac{q\eta}{p})\cap B(x_0-\frac{r'q}{p}y_0+Z,\frac{qr'}{p}+\frac{q\eta}{p})$, i.e., $\|x_0\pm \frac{r'q}{p}y_0-y+Z\|\leq \frac{qr'}{p}+\frac{q\eta}{p}$. Let $z_i\in Z$, $i=1,2,$ such that $\|x_0\pm \frac{r'q}{p}y_0-y+z_i\|\leq \frac{qr'}{p}+\frac{q\eta}{p}+\eta$. 
	
	Consider the balls $B_1=B(x_0+\frac{r'q}{p}y_0-y,\frac{qr'}{p}+\frac{q\eta}{p}+\eta)$ and
	
    $B_2=B(x_0-\frac{r'q}{p}y_0-y,\frac{qr'}{p}+\frac{q\eta}{p}+\eta)$.
	
	Because $B_1\cap B_2\neq \phi$ and $B_i\cap Z\neq \phi$ $(i=1, 2)$, from the $2$-ball property of $Z$, we obtain
	\[
	 \exists z\in Z\cap B\left(x_0+\frac{r'q}{p}y_0-y,\frac{qr'}{p}+\frac{q\eta}{p}+2\eta\right)\cap B\left(x_0-\frac{r'q}{p}y_0-y,\frac{qr'}{p}+\frac{q\eta}{p}+2\eta\right).
	\]
	Thus, $y+z\in Y\cap B(x_0+\frac{r'q}{p}y_0,\frac{qr'}{p}+\frac{q\eta}{p}+2\eta)\cap B(x_0-\frac{r'q}{p}y_0,\frac{qr'}{p}+\frac{q\eta}{p}+2\eta)$. As $\frac{r'q}{p}>1$ and $\frac{q\eta}{p}+2\eta<\e$, the {\sc claim} follows for $R\geq \frac{r'q}{p}$. 
	
	This completes the proof.
\end{proof}

We now discuss the cases of vector-valued functions. A contra-positive argument leads to the following hypothesis. Let $(\Omega, \Sigma, \mu)$ be a probability space and $Y$ be a subspace of $X$ such that $X^*$ has the Radon--Nikod$\acute{y}$m property (see \cite[pg. 61]{DU} for the definition). We denote this property by RNP in the remainder of this paper. 

\begin{thm}\label{T8}
If $L_p(\mu, Y)$ has property-$(wU)$ in $L_p(\mu, X)$, then $Y$ has property-$(wU)$ in $X$, where $1< p<\iy$.
\end{thm}

Furthermore, we have the following.

\begin{thm}\label{T7}
	Let $Y$ be a subspace of $X$, and let $X^*$ has the RNP. If $C(K,Y)$  has property-$(wU)$  in $C(K,X)$, then $Y$ has property-$(wU)$ in $X$.
\end{thm}

We do not know the validity of the converse of Theorem~\ref{T8}, \ref{T7}.

\section{On property-$(SU)$ and property-$(HB)$}
\subsection{Characterizations} We begin this section with the following characterizations of property-$(U)$, $(SU)$, and $(HB)$.

\begin{thm}\label{T4}
	Let $Y$ be a subspace of $X$; then, the following are equivalent:
	\bla
	\item $Y$ has property-$(U)$ in $X$.
	\item For every subspace $Z (\neq 0)$ of $Y$, $Y/Z$ has property-$(U)$ in $X/Z$.
	\item For every hyperplane $Z$ in $Y$, $Y/Z$ has property-$(U)$ in $X/Z$.
	\el
\end{thm}

\begin{proof}
	From \cite[Theorem~4.1$(a)$]{DPR}, we have that $(a)\Ra (b)$ and $(b)\Ra (c)$ are obvious. We only need to show that $(c)\Ra (a)$.
	
	Let us assume that $Y$ does not have property-$(U)$ in $X$. Hence, there exists $y^*\in Y^*$ and distinct $x_1^*, x_2^*\in HB(y^*)$. Then, there exists $z\in X$ such that $x_1^*(z)\neq x_2^*(z)$. Define $Z=\ker (y^*)$, then $Z\ci Y$ is a hyperplane in $Y$ and $(Y/Z)^*=\textrm{span} \{y^*\}$. The spanning vector $y^*$ possesses two distinct norm-preserving extensions over $(X/Z)$. Hence, the result follows.
\end{proof}

\begin{rem}
	\bla
\item Note that the condition $(c)$ in Theorem~\ref{T4} can be restated as follows: for a hyperplane $Z$ in $Y$ and $y\in Y\sm Z$, $y+Z$ is a smooth point of $X/Z$. 
\item From the previous remark and Theorem~\ref{T4} it follows that if $c_0\ci Y\ci \ell_\infty$ where $\dim (Y/c_0)=1$ then $Y$ can not have property-$(U)$ in $\ell_\infty$.
\el
\end{rem}

Our next theorem characterizes property-$(SU)$ in Banach spaces. 

\begin{thm}\label{T6}
	Let $Y$ be a subspace of $X$: then, the following are equivalent:
	\bla
	\item $Y$ has property-$(SU)$ in $X$.
	\item $Y$ has property-$(U)$ in $X$ and $Y$ is an ideal in $X$.
	\item $Y^\perp$ is Chebyshev in $X^*$, and the metric projection $P_{Y^\perp}:X^*\ra Y^\perp$ is a linear projection.
	\el
\end{thm}

\begin{proof}
	 $(a)\Lglra (b)$ follows from \cite{O}. 
	 
	 $(a)\Ra (c)$. Let us recall from \cite{O} that, $(a)$ holds if and only if $Y^\#$ is a linear subspace.
	 Let $x_1^*,x_2^*\in X^*$. It is sufficient to show that $P_{Y^\perp}(x_1^*+x_2^*)=P_{Y^\perp}(x_1^*)+P_{Y^\perp}(x_2^*)$. Let $y_1^\perp,y_2^\perp\in Y^\perp$ such that $P_{Y^\perp}(x_1^*)=y_1^\perp$ and $P_{Y^\perp}(x_2^*)=y_2^\perp$. Consequently, for $i=1, 2$ $\|x_i^*-y_i^\perp\|=d(x_i^*,Y^\perp)=\|x_i^*|_Y\|=\|(x_i-y_i^\perp)|_Y\|$. Thus, $x_1^*-y_1^\perp,x_2^*-y_2^\perp\in Y^\#$ and hence by $(a)$, $x_1^*+x_2^*-y_1^\perp-y_2^\perp\in Y^\#$ and $\|x_1^*+x_2^*-y_1^\perp-y_2^\perp\|\leq \|x_1^*+x_2^*-y_1^\perp-y_2^\perp+y^\perp\|$ for all $y^\perp\in Y^\perp$. Hence, $\|x_1^*+x_2^*-y_1^\perp-y_2^\perp\|\leq d(x_1^*+x_2^*-y_1^\perp-y_2^\perp,Y^\perp)=d(x_1^*+x_2^*,Y^\perp)$. Therefore, $P_{Y^\perp}(x_1^*+x_2^*)=y_1^\perp+y_2^\perp=P_{Y^\perp}(x_1^*)+P_{Y^\perp}(x_2^*)$.
	 
	 To prove $(c)\Ra (b)$, we show that $Y^\perp$ is the kernel of a contraction. The proof of $\|I-P_{Y^\perp}\|=1$ remains. Let $x^*\in X^*$ and $x^*=y^\perp+z^*$. Then, $\|z^*\|=\|x^*-y^\perp\|=d(x^*,Y^\perp)\leq \|x^*\|$, and hence, the result follows.
\end{proof}

Focusing on property-$(HB)$, we have the following.

\begin{thm}\label{T1}
	Let $Y$ be a subspace of $X$. Then, the following are equivalent.
	\bla
	\item $Y$ has property-$(HB)$ in $X$.
	\item $Y$ has property-$(U)$ in $X$, and there exists a projection $P:X^*\ra Y^\perp$ such that $\|P\|=1=\|I-P\|$.
	\item $Y^\perp$ is Chebyshev in $X^*$, and the metric projection $P_{Y^\perp}:X^*\ra Y^\perp$ is a linear projection of norm-$1$.
	\el
\end{thm}
\begin{proof}
	$(a)\Lglra (b)$. Clearly, we have $(a)\Ra (b)$. By contrast, it is sufficient to show that $\|(I-P)(x^*)\|<\|x^*\|$ whenever $(I-P)(x^*)\neq x^*$. Suppose that is not the case, then there exists $x_0^*\in X^*$ such that $(I-P)(x_0^*)\neq x_0^*$ and $\|(I-P)(x_0^*)\|=\|x_0^*\|$. Clearly, $\|x_0^*|_Y\|=\|(I-P)x_0^*|_Y\|\leq \|(I-P)x_0^*\|$. Let $\widetilde{x_0^*}$ be a norm-preserving extension of $x_0^*|_Y$. Then $x_0^*-\widetilde{x_0^*}\in Y^\perp$; hence, $(I-P)x_0^*=(I-P)\widetilde{x_0^*}$. Consequently, $\|(I-P)x_0^*\|=\|(I-P)\widetilde{x_0^*}\|\leq \|\widetilde{x_0^*}\|=\|x_0^*|_Y\|$. Thus $(I-P)(x_0^*),x_0^*\in HB(x_0^*|_Y)$, which is a contradiction.
	
	$(c)\Ra (b).$ If $(c)$ holds, then for any $x^*\in X^*$, there exists $y^\perp\in Y^\perp$ and $y^*\in X^*$ such that $x^*=y^\perp+y^*$. Now, we have $\|y^*\|=\|x^*-y^\perp\|=d(x^*,Y^\perp)\leq \|x^*\|$. Hence, both $P_{Y^\perp}$ and $I-P_{Y^\perp}$ are of norm-$1$ and $(b)$ follows. 
	
	$(a)\Ra (c).$ Clearly, $Y$ has property-$(U)$ in $X$. Now, if $(a)$ holds, then $X^*=Y^\perp\bigoplus Y^\#$. Now for every $x^*\in X^*$, there exists a unique $y^\perp\in Y^\perp$ and $y^\#\in Y^\#$ such that $x^*=y^\perp+y^\#$ where $\|y^\perp\|\leq \|x^*\|$. Hence, $\|x^*-y^\perp\|=\|y^\#\|=\|y^\#|_Y\|=\|x^*|_Y\|=d(x^*,Y^\perp)$. Therefore, $y^\perp=P_{Y^\perp}(x^*)$, which completes the proof.
\end{proof}

Our next observation serves a sufficient condition for property-$(HB)$.

\begin{prop}
	Let $Y$ be a subspace of $X$. If $Y^{\perp\perp}$ has property-$(HB)$ in $X^{**}$, then $Y$ has property-$(HB)$ in $X$.
\end{prop}

\begin{proof}
	Clearly, $Y$ has property-$(U)$ in $X$ if $Y^{\perp\perp}$ has property-$(U)$ in $X^{**}$. In fact, if $y^*\in Y^*$ and distinct $x_1^*,x_2^*\in HB(y^*)$, then distinct $J_{X^*}(x_1^*),J_{X^*}(x_2^*)\in HB(J_{Y^*}(y^*))$, where $J_{X^*}:X^*\hookrightarrow X^{***}$ is the canonical embedding. It has been shown in \cite[pg. 11]{DPR} that $Y$ is ideal in $X$ if and only if $Y^{\perp\perp}$ is ideal in $X^{**}$. Hence, $Y$ has property-$(SU)$ in $X$ if $Y^{\perp\perp}$ has property-$(HB)$ in $X^{**}$, and we have the decomposition $X^*=Y^{\perp}\bigoplus Y^{\#}$.
	
	Owing to the property-$(HB)$ of $Y^{\perp\perp}$ in $X^{**}$, there exists $Q:X^{***}\ra X^{***}$, which is a bi-contractive projection with $\ker (Q)=Y^{\perp\perp\perp}$.  As $Y$ has property-$(SU)$ in $X$, there exists  a projection $P:X^*\ra Y^{\#}$ with $\|P\|=1$. Clearly, $P^{**}:X^{***}\ra Y^{***}$ is a contractive projection with $\ker (P^{**})=Y^{\perp\perp\perp}$, which leads to $Q=P^{**}$.
	
	Now, if there exists $x^*\in X^*$ such that $\|(I-P)x^*\|>\|x^*\|$ then  $\|(I-Q)x^*\|=\|(I-P^{**})x^*\|>\|x^*\|$. Property-$(HB)$ of $Y^{\perp\perp}$ in $X^{**}$ is contradicted.
\end{proof}

\subsection{On subspaces of tensor product spaces} We now identify a few cases mainly in the spaces of vector valued functions and tensor product spaces, where property-$(HB)$ remains stable.
By $(\Omega,\Sigma,\mu)$ we mean a probability space, and let $L_p(\mu,X)$ be the Banach space consisting of all $p$th Bochner integrable functions (see \cite{DU} for details) as defined in Section~2. It is well known that $L_p(\mu,X)^*\cong L_q(\mu,X^*)$, $\frac{1}{p}+\frac{1}{q}=1$ if and only if $X^*$ has the RNP.

A routine verification of the property-$(HB)$ in the proofs of \cite[Theorem~3.4, 3.6]{DPR} allows us to prove the following Theorems ~\ref{T9} and \ref{T10}.

\begin{thm}\label{T9}
Let $Y$ be a subspace of $X$, and let $X^*$ has the RNP. Then, $L_p(\mu,Y)$ has property-$(HB)$ in  $L_p(\mu,X)$ if and only if $Y$ has property-$(HB)$ in $X$, for $1< p<\infty$.
\end{thm}

  We denote the injective (projective) tensor product for Banach spaces $X$ and $Z$ by $X\widehat{\otimes}_\e Z$ ($X\widehat{\otimes}_\pi Z$) (see \cite{R} for details).  We denote the injective (projective) norm of $\varphi\in X\widehat{\otimes}_\e Z$ ($\varphi\in X\widehat{\otimes}_\pi Z$) by $\e(\varphi)$ ($\pi (\varphi)$).
  It is well known that if one of the spaces $X^*, Z^*$ has the RNP and one of them has the {\it approximation property}, $\left(X\widehat{\otimes}_\e Z\right)^*\cong X^*\widehat{\otimes}_\pi  Z^*$ (see \cite[Pg. 114]{R}). In the case of the projective tensor product, we have $(X\widehat{\otimes}_\pi Z)^*\cong \mc{L}(X,Z^*)$: the space of all bounded linear operators from $X$ to $Z^*$ (see \cite[Pg.24]{R}).
  Let us also recall that if $Y$ is a subspace of $Z$, then $X\widehat{\otimes}_\e Y$ is a subspace of $X\widehat{\otimes}_\e Z$; by contrast, $X\widehat{\otimes}_\pi Y$ is not necessarily a subspace of $X\widehat{\otimes}_\pi Z$ in general. $X\widehat{\otimes}_\pi Y$ is a  subspace of $X\widehat{\otimes}_\pi Z$ if $Y$ is an ideal in $Z$ (see \cite[Theorem 1.(i)]{RA}).

A routine generalization for property-$(HB)$ of \cite[Theorem~3.6]{DPR} leads to the following.

\begin{thm}\label{T10}
Let $X$ be a Banach space such that $X^*$ has the RNP.
Let $E$ be an $L_1$-predual space and let $Y$ be a subspace of $X$. Then, $E\widehat{\otimes}_\e Y$ has property-$(HB)$ in $E\widehat{\otimes}_\e X$ if and only if $Y$ has property-$(HB)$ in $X$.
\end{thm}

It is well known that $C(K,X)\cong C(K)\widehat{\otimes}_\e X$ for any Banach space $X$. 

\begin{cor}
Let $Y$ be a subspace of $X$, and let $X^*$ has the RNP. Then, $Y$ has property-$(HB)$ in $X$ if and only if $C(K,Y)$ has property-$(HB)$ in $C(K,X)$. 
\end{cor}

\begin{rem}\label{R1}
	As stated earlier, if $Y$ is an ideal of a Banach space $Z$, then the projective tensor project $Y\widehat{\otimes}_\pi X$ is a subspace of $Z\widehat{\otimes}_\pi X$. In addition, if $X$ is a Banach space, then $L_1(\mu)\widehat{\otimes}_\pi X\cong L_1(\mu, X)$ (see \cite[Example 2.19]{R}). Along with \cite[Example~2.3]{BR}, one can conclude that property-$(U)$ may not be stable under a projective tensor product.
\end{rem}

However, we can identify a class of functionals in $(X\widehat{\otimes}_{\pi}Y)^*$ that still have unique norm-preserving extensions over $X\widehat{\otimes}_{\pi}Z$.

\begin{thm}\label{T12}
	Let $X$ and $Z$ be Banach spaces, and $Y$ be a subspace of $Z$. If $Y$ has property-$(SU)$ in $Z$, then every isometry in $\mc{L}(X,Y^*)\cong (X\widehat{\otimes}_{\pi}Y)^*$ has a unique norm-preserving extension to $(X\widehat{\otimes}_{\pi}Z)$. 
\end{thm}

\begin{proof}
	Let $S\in \mc{L}(X,Y^*)$ be an isometry and $T_1,T_2$ be two distinct norm-preserving extensions of $S$ to $(X\widehat{\otimes}_{\pi}Z)$. Let $x_0\in S_X$ such that $T_1x_0\neq T_2x_0$. As $T_1|_{(X\widehat{\otimes}_{\pi}Y)}=S=T_1|_{(X\widehat{\otimes}_{\pi}Y)}$, we have $T_1(x_0\otimes y)=S(x_0\otimes y)=T_2(x_0\otimes y)$, that is, $T_1x_0(y)=Sx_0(y)=T_2x_0(y)$ for all $y\in Y$. Hence, $T_1x_0$ and $T_2x_0$ are two linear extensions of $Sx_0$. Because $S$ is an isometry, we have $\|Sx_0\|=\|x_0\|=1$. Now, $\|S\|=\|T_i\|\geq \|T_ix_0\|\geq \|Sx_0\|=\|x_0\|=1=\|S\|$, $i=1,2$. Consequently, $\|T_1x_0\|=\|T_2x_0\|=\|Sx_0\|$. Hence, $T_1x_0$ and $T_2x_0$ are two distinct norm-preserving extensions of $Sx_0$.
\end{proof}

We have an affirmative answer to the converse of Theorem~\ref{T12}.

\begin{thm}\label{T13}
	Let $X$ and $Z$ be Banach spaces, and $Y$ be a subspace of $Z$ such that $(X\widehat{\otimes}_{\pi}Y)$ is a subspace of $(X\widehat{\otimes}_{\pi}Z)$. If $(X\widehat{\otimes}_{\pi}Y)$ has property-$(U)$ in $(X\widehat{\otimes}_{\pi}Z)$, then $Y$ has property-$(U)$ in $Z$.
\end{thm}

\begin{proof}
	Let $y^*\in S_{Y^*}$ and distinct $z_1^*,z_2^*\in HB(y^*)$. Choose $x^*\in S_{X^*}$ and define the following maps:
	$$ S(x)=x^*(x)y^* \text{ for } x\in X \mbox{~and}$$
	$$ T_1(x)=x^*(x)z_1^*,  ~
	 T_2(x)=x^*(x)z_2^* \text{ for } x\in X.$$
	One can check that $S\in \mc{L}(X,Y^*), T_1,T_2\in \mc{L}(X,Z^*)$ and $\|S\|=\|T_1\|=\|T_2\|$. We consider $S\in (X\widehat{\otimes}_{\pi}Y)^*$ and $T_1,T_2\in (X\widehat{\otimes}_{\pi}Z)^*$. For any $x\otimes y\in X\otimes_\pi Y$, we have $S(x\otimes y)=Sx(y)=x^*(x)y^*(y)=x^*(x)z_1^*(y)=T_1x(y)=T_1(x\otimes y)$, and we obtain $S(x\otimes y)=T_2(x\otimes y)$. If $D$ is the linear span of all simple tensors in $X\otimes_\pi Y$, then $T_1|_D=T_2|_D=S$. Based on the density of $D$ in $(X\widehat{\otimes}_{\pi}Y)$, we have $T_1|_{(X\widehat{\otimes}_{\pi}Y)}=T_2|_{(X\widehat{\otimes}_{\pi}Y)}=S$. Hence, distinct $T_1,T_2\in HB(S)$.
\end{proof}

\subsection{On quotient spaces and higher duals of Banach spaces} Let $W,Y$ be subspaces of $X$ such that $W\ci Y\ci X$, where $Y$ has property-$(SU)$ in $X$; then, it is shown in \cite{DPR} that $Y/W$ has property-$(SU)$ in $X/W$. One can check $(Y/W)^\#=\{w^\perp\in W^\perp_{X^*}:\|w^\perp|_Y\|=\|w^\perp\|\}$. 

\begin{thm}\label{T3}
Let $W, Y$ be subspaces of $X$ where $W\ci Y\ci X$. If $Y$ has property-$(HB)$ in $X$, then $Y/W$ has property-$(HB)$ in $X/W$. 
\end{thm}

\begin{proof}
We have that $Y/W$ has property-$(SU)$ in $X/W$. Let $x^*\in W^{\perp}_{X^*}(=(X/W)^*)$ and $x^*=y^\#+y^\perp$, where $y^\#\in (Y/W)^{\#}$ and $y^\perp\in Y^\perp_{W^\perp}(=(Y/W)^\perp)$. It remains to show $\|x^*\|\geq \|y^\perp\|$. Since $(Y/W)^\#=\{w^\perp\in W^\perp_{X^*}:\|w^\perp|_Y\|=\|w^\perp\|\}$ and $y^\#\in (Y/W)^{\#}$, we have $y^\#\in Y^{\#}$. Moreover, $y^\perp\in Y^{\perp}_{W^\perp}\ci Y^{\perp}_{X^*}$; we can view $x^*=y^\#+y^\perp$, where $y^\#$ and $y^\perp$ are obtained from $Y^{\#}$ and $Y^{\perp}$, respectively. Hence, $\|x^*\|\geq \|y^\perp\|$ as $Y$ has property-$(HB)$ in $X$.
\end{proof}

The converse of Theorem~\ref{T3} is true under the assumption that $W$ is an $M$-ideal in $X$. 

\begin{thm}
Let $J$ and $Y$ be subspaces of $X$ with $J\ci Y\ci X $, and let $J$ be an $M$-ideal in $X$. If $Y/J$ has property-$(HB)$ in $X/J$, then $Y$ has property-$(HB)$ in $X$.
\end{thm}

\begin{proof}
First, we have that $Y$ has property-$(SU)$ in $X$ (see \cite[Theorem~4.5]{DPR}). Consider $x^*\in X^*$ and $x^*=y^\#+y^\perp$, where $y^\#\in Y^\#$, $y^\perp \in Y^\perp$. Our aim is to show that $\|x^*\|\geq \|y^\perp\|$. Let $\widetilde{x^*}\in X^*$ be the unique norm-preserving extension of $x^*|_J$. Clearly, $x^*-\widetilde{x^*},y^\#-\widetilde{x^*},y^\perp \in J^{\perp}$. Moreover, as $\widetilde{x^*}\in J^\#$, we have $\widetilde{x^*}\in Y^\#$; hence, $y^\#-\widetilde{x^*}\in Y^\#$. Thus, $x^*-\widetilde{x^*}=y^\#-\widetilde{x^*}+y^\perp$, where $x^*-\widetilde{x^*}\in J^\perp(=(X/J)^*)$, $y^\#-\widetilde{x^*}\in (Y/J)^\#(=Y^\#_{J^\perp})$ and $y^\perp\in Y^\perp_{J^\perp}(=(Y/J)^\perp)$. Consequently, $\|x^*-\widetilde{x^*}\|\geq \|y^\perp\|$ as $Y/J$ has property-$(HB)$ in $X/J$. By contrast, $x^*=x^*-\widetilde{x^*}+\widetilde{x^*}$, where $x^*-\widetilde{x^*}\in J^\perp$, $\widetilde{x^*}\in J^\#$, and $J$ is an $M$-ideal in $X$, implying that $\|x^*\|\geq \|x^*-\widetilde{x^*}\|$ and finally $\|x^*\|\geq \|y^\perp\|$. 
\end{proof}

Combining the cases for an $L_1$-predual or an $M$-embedded space, the following can be obtained:

\begin{thm}
Let $X$ be an $L_1$-predual (or an $M$-embedded) space, and let $Y$ be a finite co-dimensional subspace of $X$. Then, $Y^{\perp\perp}$ has property-$(HB)$ in $X^{**}$ if $Y$ has property-$(HB)$ in $X$.
\end{thm}

\begin{proof}
 Let $Y$ have property-$(HB)$ in $X$ and $P$ be the corresponding bi-contractive projection on $X^*$ with $\ker (P)=Y^\perp$. Thus $P^{**}$ is a bi-contractive projection with $\ker (P^{**})=Y^{\perp\perp\perp}$. By Theorem~\ref{T1}, it is sufficient to prove that $Y^{\perp\perp}$ has property-$(U)$ in $X^{**}$.
 
  In both cases, $X^*$ is the $L$-summand in $X^{***}$. Since $Y^\perp(= \textrm{span} \{x_i^*:1\leq i\leq n\})$ is Chebyshev in $X^*$, $Y^{\perp\perp\perp}=\textrm{span} \{\widehat{x_i^*}:1\leq i\leq n\}$ is Chebyshev in $X^{***}$, where $\widehat{x_i^*}$'s are the canonical images of $x_i^*$'s in $X^{***}$.
 Hence $Y^{\perp\perp}$ has property-$(U)$ in $X^{**}$.
\end{proof}

Our next observation follows from the definition of property-$(U)$ ($(SU)$). 

\begin{prop}\label{37}
Let $Z$ and $Y$ be the subspaces of $X$ with $Z \ci Y \ci X$. \bla
 \item If $Z$ has property-$(SU)$ in $X$, then $Z$ has property-$(SU)$ in $Y$. 
 \item If $Z$ (and $Y$) has property-$(SU)$ in $Y$ (and $X$), then $Z$ has property-$(SU)$ in $X$.
 \el
\end{prop}

\begin{thm}\label{T28}
Let $Z$ and $Y$ be subspaces of $X$ with $Z \ci Y \ci X$. If $Z$ has property-$(HB)$ in $X$, then $Z$ has property-$(HB)$ in Y.
\end{thm}

\begin{proof}
Suppose that $Z$ has property-$(HB)$ in $X$. From Proposition \ref{37}, $Z$ has property-$(SU)$ in $Y$. Let $y^*\in Y^*$ and $x^*\in HB(y^*)$. Let $y^*=y^{\#}+y^{\perp}$, where $y^{\#}\in Z^{\#}_{Y^*}$ and $y^{\perp}(\neq 0)\in Z^{\perp}_{Y^*}$. It remains to be shown that $\|y^*\|\geq \|y^{\perp}\|$. Because $Z$ has property-$(HB)$ in $X$, we have $\|x^*\|>\|x^{\#}\| $ and $\|x^*\|\geq \|x^{\perp}\|$, where $x^*=x^{\#}+x^{\perp}$, $x^{\#}\in Z^{\#}_{X^*}$, $x^{\perp}(\neq 0)\in Z^{\perp}_{X^*}$. Now, $y^*=x^*|_Y=x^{\#}|_Y+x^{\perp}|_Y$. Since $x^{\perp}|_Y \in Z^{\perp}_{Y^*}$ and $x^\#|_Y\in Z^\#_{Y^*}$, because of the uniqueness of the decomposition of $y^*$, we have $y^{\#}=x^{\#}|_Y$ and $y^{\perp}=x^{\perp}|_Y$. Hence, $\|y^*\|=\|x^*\|\geq \|x^{\perp}\|\geq \|x^{\perp}|_Y\|=\|y^{\perp}\|$.
\end{proof}

\begin{rem}
Example~\ref{S1} shows that even if $Y$ has property-$(SU)$ in $X$, the converse of Theorem~\ref{T28} may not be true. 
\end{rem}

Suppose that $X$ has property-$(U)$ in $X^{**}$. Our next theorem provides a sufficient condition for subspaces to have property-$(HB)$ in their biduals. Note that $Y$ has property-$(HB)$ in $X$ if $Y$ has property-$(HB)$ in $X^{**}$, following from Theorem~\ref{T28}.

\begin{thm}
Let $Y$ be a subspace of $X$, and let $X$ have property-$(U)$ in $X^{**}$. If $Y$ has property-$(HB)$ in $X$, then $Y$ has property-$(HB)$ in $Y^{**}$.
\end{thm}

\begin{proof}
Clearly, $Y$ has property-$(SU)$ in $Y^{**}$. Let $y^{***}\in Y^{***}$ and suppose $y^{***}=y^\#+y^\perp$ where $y^\#\in Y^\#_{Y^{***}}$ and $y^\perp\in Y^\perp_{Y^{***}}$. It remains to be shown that $\|y^{***}\|\geq \|y^\perp\|$. Let $P:X^*\ra Y^\#$ be a bi-contractive projection with $\ker (P)=Y^\perp$. Hence, $P^{**}:X^{***}\ra Y^{\#\perp\perp}$ is a bi-contractive projection with $\ker (P^{**})=Y^{\perp\perp\perp}$. As $Y$ has property-$(HB)$ in $X$, $Y^*$ and $Y^{***}$ are subspaces of $X^*$ and $X^{***}$, respectively. Thus, $y^{***}\in X^{***}$, $y^\#\in Y^\#_{Y^{***}}\cong Y^*\cong Y^{\#}_{X^*}\ci Y^{\#\perp\perp}_{X^{***}}$ and $y^\perp\in Y^\perp_{Y^{***}}\ci  Y^{\perp\perp\perp}_{X^{***}}$. Hence, $\|y^\perp\|= \|(I-P^{**})y^{***}\|\leq \|y^{***}\|$.
\end{proof}

\subsection{Examples} We begin with the following example, which ensures that the property-$(wU)$ is strictly weaker than the property-$(U)$.

\begin{ex}
Let $X$ be a non-reflexive Banach space such that $X^*$ is separable. Let $\mathcal{N}(X)$ be the set of all
equivalent norms on $X$ , equipped with the topology of uniform convergence on bounded subsets
of X . Suppose $M$ be a Polish space and $A$ and $B$ be two analytic subsets of $M$ with $A\ci B$.
By \cite[Theorem~3]{GYK}, there exists a continuous map $\Phi:M\ra \mathcal{N}(X)$ such that $(X,\Phi(t))$ is smooth Banach space but it dual is not a strictly convex Banach space, for $t\in B\sm A$, where $(X,\Phi(t))$ denotes the Banach $X$ endowed with the norm $\Phi(t)$.

Since the dual space $(X^*,\Phi(t)^*)$ is not strictly convex there exists $x^*\in X^*$ such that $\{tx^*:t\in\mb{R}\}$ is not a Chebyshev subspace. Hence $ker (x^*)$ can not have property-$(U)$ in $(X,\Phi(t))$ but any subspace has property-$(wU)$ in $(X,\Phi(t))$ (see \cite[Theorem~2.4]{LA}).
\end{ex}

We now give a few examples that have property-$(SU)$ but not $(HB)$. 

\begin{thm}\label{T17}
	Let $(a_n)\in S_{\ell_1}$ such that $1>\sup_{n\in \mathbb{N}}|a_n|>\frac{1}{2}$. Then, $\ker((a_n))\ci c_0$ has property-$(SU)$ but does not have property-$(HB)$ in $c_0$.
\end{thm}

\begin{proof}
	Let $X=c_0$, $Y=\ker ((a_n))$. Because $\sup_{n\in \mathbb{N}}|a_n|>\frac{1}{2}$, there exists a unique $N\in \mathbb{N}$ such that $|a_N|>\frac{1}{2}$. By \cite[Theorem 6.6]{DPR}, the metric projection $P_{Y^\perp}:X^*\ra Y^{\perp}$ is defined as $$P_{Y^\perp}((x_n))=\{\frac{x_N}{a_N}(a_1,a_2,a_3,\ldots)\}\text{ for }(x_n)\in \ell_1.$$ Clearly, $P_{Y^\perp}$ is linear and $\|I-P_{Y^\perp}\|=1$. If we choose $(x_n)\in S_{\ell_1}$ such that $|x_N|>|a_N|$, then $\|P_{Y^\perp}(x_n)\|>1=\|(x_n)\|_1$, i.e., $\|P_{Y^\perp}\|>1$. Hence, by Theorem~\ref{T1}.$(c)$, $Y$ does not have property-$(HB)$ in $c_0$.
\end{proof}

From Theorem~\ref{T17} it is clear that if $\ker((a_n))$ has property-$(HB)$ in $c_0$ then $\sup_{n\in \mathbb{N}}|a_n|=1$ which in turn $(a_n)=e_m$ for some $m$. Hence we have the following.

\begin{cor}
	Let $Y$ be a hyperplane in $c_0$. $Y$ has property-$(HB)$ in $c_0$ if and only if $Y$ is an $M$-summand in $c_0$.
\end{cor}

\begin{ex}\label{E2}
Let $m,n\in \mathbb{N}$, and $m<n$. Then, the subspace $\textrm{span}~\{e_1,\dots,e_m\}$ has property-$(SU)$ but does not have property-$(HB)$ in $(\mathbb{R}^n,\|.\|_*)$, where $\|(x_i)_{i=1}^{n}\|_*=\max\{|x_1|,\dots,|x_{n-1}|,\frac{|x_1+\dots+x_n|}{n}\}$ for $(x_i)_{i=1}^{n}\in \mathbb{R}^n$.
\end{ex}

\begin{proof}
Let $Y=\textrm{span}~\{e_1,\dots,e_m\}$ and $X=(\mathbb{R}^n,\|.\|_*)$. It is easy to observe that $X^*=(\mathbb{R}^n,\|.\|^*)$, where
$\|(a_i)_{i=1}^n\|^*=\sum_{i=1}^{n-1}|a_i-a_n|+n|a_n|$ for $(a_i)_{i=1}^n\in \mathbb{R}^n$.

Let $G=\textrm{span}~\{e_1,\dots,e_m\}$ be a subspace of $X^*$. It is evident that $X^*=G\oplus Y^\perp$. In fact, $(a_i)_{i=1}^n=(a_1,\dots,a_m,0,\dots,0)+(0,\dots,0,a_{m+1},\dots,a_n)$ for an element $(a)_{i=1}^{n}\in X^*$, where $(a_1,\dots,a_m,0,\dots,0)\in G$ and $(0,\dots,0,a_{m+1},\dots,a_n)\in Y^\perp$. Using the triangle inequality, we have $\|(a_1,\dots,a_m,0,\dots,0)\|^*\leq \|(a_1,\ldots, a_n)\|^*$ and obtain a strict inequality if $a_j$'s are not all zeros for $m< j\leq n$. Thus, $Y$ has property-$(SU)$ in $X$.

Consider the decomposition $(b_i)_{i=1}^n=(b_1,\dots,b_m,0,\dots,0)+(0,\dots,0,b_{m+1},\dots,b_n)$, where $b_i=1$, $1\leq i\leq n$, but $\|(b_i)_{i=1}^n\|^*=n<n+m=\|(0,\dots,0,b_{m+1},\dots,b_n)\|^*$. Hence, $Y$ does not have property-$(HB)$.
\end{proof}

 We conclude this section by showing that property-$(HB)$ cannot be transferred via property-$(SU)$ in general. 
  
\begin{ex}\label{S1}
 Let $X=(\mathbb{R}^3,\|.\|_\iy)$ and $Y=\{(x,y,z)\in \mathbb{R}^3:x+y+6z=0\}$, then $Y^{\#}=\{(x,y,0):x,y\in \mathbb{R}\}\ci (\mathbb{R}^3,\|.\|_1)$. Because $Y^{\#}$ is linear, $Y$ has property-$(SU)$ in $X$. Let $Z=\ker ((1,0,0))\ci Y$, i.e., $Z=\{(0,y,z):y+6z=0\}$. Clearly, $Z_{Y^*}^\perp=\emph{span}\{(1,0,0)\}$ is an $L$-summand in $Y^*$. Hence, $Z$ is an M-ideal in $Y$, and $Y$ has property-$(SU)$ in $X$. It is clear that plane $Z$ does not intersect the upper and bottom surfaces of the unit ball $B_{(\mathbb{R}^3,\|.\|_\infty)}$.  Thus, $\emph{ext}(B_Z)=\{\pm(0,1,-\frac{1}{6})\}$, and we can show that $Z^{\#}=\{(0,y,0):y\in \mathbb{R}\}$. An easy calculation indicates that $Z^{\perp}_{X^*}=\emph{span}\{(1,0,0),(0,1,6)\}$. Let $(0,0,1)\in X^*$. Then, the unique decomposition of $(0,0,1)$ is $$(0,0,1)=(0,-\frac{1}{6},0)+(0,\frac{1}{6},1),$$
 and $\|(0,\frac{1}{6},1)\|_1=\frac{7}{6}>1=\|(0,0,1)\|_1$. Hence, $Z$ does not have property-$(HB)$ in $X$.
\end{ex}

\section{On property-$(k-U)$ in Banach spaces}

\subsection{Characterization} We derive a characterization for reflexive subspaces with property-$(k-U)$ in Theorem~\ref{T5}. Let $x\in X\sm\{0\}$ and $x^*\in X^*\sm \{0\}$. Let us recall from Section ~1 that
$S(X^*,x)=\{x^*\in S_{X^*}:x^*(x)=\|x\|\}$; we also define  
 $S(X,x^*)=\{x\in S_X:x^*(x)=\|x^*\|\}$. 

\begin{prop}{\label{T2}}
	Let $Y$ be a subspace of $X$. If $Y$ is a $k$-Chebyshev subspace of $X$, then $ \dim(\{S(X,y^\perp)- x\}\cap Y)< k$ for all $x\in S(X,y^\perp)$, for all $y^\perp\in S_{Y^\perp}$.
\end{prop}
\begin{proof} 
	If possible, let $y^\perp\in S_{Y^\perp}$ such that $\dim (\{S(X,y^\perp)- x\}\cap Y)\geq k$ for some $x\in S(X,y^\perp)$. Let $y_1,y_2,\dots,y_k\in \{S(X,y^\perp)- x\}\cap Y$ be $k$ linearly independent elements. Since $x,y_i+x\in S(X,y^\perp)$, $1\leq i\leq k$, we have $\|y_i+x\|=1=\|x\|$. Hence, for $y\in Y$, $\|x-y\|\geq y^\perp(x-y)=y^\perp(x)=\|y^\perp\|=1=\|x\|$, which is true for all $y\in Y$. Consequently, $d(x,Y)=\|x\|=\|x+y_i\|$, that is, $0,-y_i\in P_Y(x)$ for $1\leq i\leq k$. Therefore, $Y$ is not a $k$-Chebyshev subspace of $X$.
\end{proof}

The converse of the aforementioned theorem is true for the proximal subspace of $X$. We first prove a similar result for the hyperplanes in $X$.

\begin{prop} \label{P2}
	Let $Y=\ker (x^*)$ be a proximinal hyperplane of $X$ for some $x^*\in S_{X^*}$. If $Y$ is not a $k$-Chebyshev subspace of $X$, then $\dim (\{S(X,x^*)- x\}\cap Y)\geq k$ for some $x\in S(X,x^*)$.
\end{prop}
\begin{proof}
	Let $Y$ not be a $k$-Chebyshev subspace of $X$. Let $x\in X$ be such that $d(x,Y)=1=\|x\|= x^*(x)$ and $\dim (P_Y(x))\geq k$. 
	
	Let $y_1,\dots,y_k\in Y$ be $k$ linearly independent elements such that $\|y_i-x\|=d(x,Y)=1$, $i=1,\dots, k $. Then, we have $-y_i\in S(X,x^*)-x$, where $x\in S(X,x^*)$. Hence, the result follows.
\end{proof}

We now generalize the aforementioned proposition for proximinal subspaces of $X$.

\begin{prop} \label{P1}
	Let $Y$ be a proximinal subspace of $X$. If $Y$ is not a $k$-Chebyshev subspace of $X$, then $\dim (\{S(X,y^\perp)- x\}\cap Y)\geq k$ for some $x\in S(X,y^\perp)$ and for some $y^\perp\in S_{Y^\perp}$.
\end{prop}
\begin{proof}
	If $Y$ is not a $k$-Chebyshev subspace of $X$, then there exists $x_0\in X\sm Y$ such that $Y$ is not a $k$-Chebyshev subspace of $Z=\textrm{span} \{Y\cup x_0\}$. Clearly, $Y$ is proximinal in $Z$, and there exists $z^*\in S_{Z^*}$ such that $Y=\ker (z^*)$. From Proposition~\ref{P2}, we have $\dim (\{S(Z,z^*)- z\}\cap Y)\geq k$ for some $z\in S(Z,z^*)$. If $x^*$ is a Hahn--Banach extension of $z^*$ over $X$, then $\dim (\{S(X,x^*)- z\}\cap Y)\geq k$, and the result follows.
\end{proof}

Hence, we have the following.

\begin{thm}\label{PP2}
Let $Y$ be a proximinal subspace of $X$. $Y$ is a $k$-Chebyshev subspace of $X$ if and only if $\dim (\{S(X,y^\perp)- x\}\cap Y)< k$ for all $x\in S(X,y^\perp)$ and for all $y^\perp\in S_{Y^\perp}$.
\end{thm}

We now state Theorem~\ref{T11} in terms of the characterization derived in Theorem~\ref{PP2}.

\begin{thm}\label{T5}
	Let $Y$ be a subspace of $X$. Consider the following statements:
	\bla
	\item $Y$ has property-$(k-U)$ in $X$.
	\item $Y^\perp$ is a $k$-Chebyshev subspace of $X^*$.
	\item $\dim (\{S(X^*,x)- x^*\}\cap Y^\perp)< k$ for all $x^*\in S(X^*,x)$ and for all $x\in  S_Y$.
	\el
	Then, $(a)\Lglra (b)$ and $(b)\Ra (c)$. If $Y$ is reflexive, then $(c)\Ra (b)$.
\end{thm}
\begin{proof}
	$(a)\Lglra (b)$ follows from \cite{X}.
	
	For $(b)\Ra (c)$, assume that there exists $x\in S_Y$ such that $\dim (\{S(X^*,x)- x^*\}\cap Y^\perp)\geq k$ for some $x^*\in S(X^*,x)$. Identifying $x$ in $Y^{**}$ with its canonical image, say $\widehat{x}$, the aforementioned inequality can be stated as $\dim (\{S(X^*,\widehat{x})- x^*\}\cap Y^\perp)\geq k$. By Theorem~\ref{T2}, $Y^\perp$ is not a $k$-Chebyshev subspace. 
	
	$(c)\Ra (b)$ follows from Theorem \ref{PP2} as $Y=Y^{\perp\perp}$ when $Y$ is reflexive.
\end{proof}

\subsection{Applications to subspaces of Bochner Integrable functions}

We assume $(\Omega, \Sigma,\mu)$ to be a measure space. 

\begin{thm}\label{T14}
Let $\mu$ be non-atomic and $Y=\emph{span}\{\chi_{A_i}: 1\leq i\leq n\}$ be an $n$-dimensional subspace of $L_1(\mu)$,  where $A_1,A_2,\dots,A_n\in \Sigma$ are disjoint, and $\mu(\Omega\sm\cup_{i=1}^{n}A_i)>0$. Then, $Y$ does not have property-$(k-U)$ for any $k\in \mathbb{N}$.
\end{thm}

\begin{proof}
Let $B=\Omega\sm\cup_{i=1}^{n}A_i$. Because $\mu(B)>0$ and $\mu$ is a non-atomic measure, there exists $B_1\ci B$ such that $\mu(B_1)<\mu(B)$. Hence, $\mu(B\sm B_1)>0$, and we find $B_2\ci B\sm B_1$ such that $0<\mu(B_2)<\mu(B\sm B_1)$. Continuing this process, we obtain $k$-disjoint sets $(B_j)_{j=1}^{k}$ of the positive measure. Let $f\in Y$; it is clear that $f=0$ a.e. on $B$. Now, $g_j:=\chi_{B_j}$ are $k$ linearly independent members of $L_\infty(\mu)$. Setting $g=\textrm{sgn} (f)$ on $\cup_{i=1}^{n}A_i$, $g=0$ a.e. on $B$. One can verify that $g_j+g,g\in S(L_1(\mu)^*,f)$ for $1\leq j\leq k$. Hence, $g_j\in S(L_1(\mu)^*,f)-g$ and $g_j\in Y^{\perp}$ as $B_j\cap (\cup A_i)=\phi$, $1\leq j\leq k$. From Theorem~\ref{T5}, it is clear that $Y$ does not have property-$(k-U)$. 
\end{proof}

\begin{prop}
Let $\mu$ be non-atomic and $Y$ be an $n$-dimensional subspace of $L_1(\mu)$.
Then, $Y$ is an $(n+1)$-Chebyshev subspace, but not a $k$-Chebyshev subspace, for $1\leq k\leq n$.
\end{prop}

\begin{proof}
An extreme point $g\in B_{L_\iy(\mu)}$ exists such that $g|_Y=0$ (see \cite[Theorem~2.5]{P}). Hence, we have $|g|=1$ a.e. and $\int_{\Omega}gfd\mu=0$ for all $f\in Y$. Let $\{f_1,\dots,f_n\}$ be a basis for $Y$ and  $h=g\sum_{i=1}^{n}|f_i|$. One can check that $d(h,Y)=\|h\|$. It is evident that $\textrm{sgn}(g\sum_{i=1}^{n} |f_i|-f_i)=\textrm{sgn}(g)=g$, as $|g|=1$. Thus, $\|h-f_i\|=\int g(g\sum_{i=1}^{n} |f_i|-f_i)d\mu=\|h\|$. Hence, $0,f_1,\dots,f_n\in P_Y(h)$. Therefore, $Y$ is not a $k$-Chebyshev subspace for $1\leq k\leq n$. 
  
  Being an $n$-dimensional subspace, $Y$ is always an $(n+1)$-Chebyshev subspace.    
\end{proof}

\begin{thm}\label{H4}
Let $Z\ci Y\ci X$ be subspaces of $X$. If $Y$ has property-$(k-U)$ in $X$, then $Y/Z$ has property-$(k-U)$ in $X/Z$.
\end{thm}

\begin{proof}
Let $Y/Z$ does not have property-$(k-U)$ in $X/Z$. Let $y^*\in Z^{\perp}_{Y^*}$ and $x^*_i\in Z^\perp_{X^*}$, and $1\leq i\leq k+1$ be an affinely independent
norm, preserving the extensions of $y^*$. As $x^*_i|_Y=y^*$ and $x_i^*$, $i=1,2,\dots,k+1$ are affinely independent norm-preserving extensions of $y^*\in Y^*$; $Y$ does not have property-$(k-U)$ in $X$. 
\end{proof}

The converse part of this result is true under the assumption that $Z$ is an $M$-ideal in $X$.

\begin{thm}\label{H5}
Let $Y, Z$ be the subspace of $X$ such that $Z\ci Y\ci X$. Suppose $Z$ is  an $M$-ideal in $X$ and $Y/Z$ has property-$(k-U)$ in $X/Z$. Then, $Y$ has property-$(k-U)$ in $X$.
\end{thm}

\begin{proof}
If possible, let $Y$ not have property-$(k-U)$ in $X$, that is, $Y^\perp$ is not a $k$-Chebyshev subspace of $X^*$. Let $x^*\in X^*$ and  $y_i^\perp, y_0^\perp\in P_{Y^\perp}(x^*)$, $i=1,2,\dots,k$, such that the vectors in  $\{y_i^\perp-y_0^\perp:1\leq i\leq k\}$ are linearly independent. As $Z$ is a  $M$-ideal  in $X$, there exists a subspace $G$ of $X^*$ such that $X^*=Z^{\perp}\oplus_{\ell_1} G$. Suppose $x^*=z^\perp+z^\#$, where $z^\perp\in Z^{\perp}$, $z^\#\in G$. 
\beqa
\mbox{Now}~ d(x^*,Y^\perp) &=& \inf_{y^\perp\in Y^\perp}\|x^*-y^\perp\| \\
&=& \inf_{y^\perp\in Y^\perp}(\|z^\perp-y^\perp\|+\|z^\#\|) \\ 
&=& d(z^\perp,Y^\perp)+\|z^\#\| \\
&\leq & \|z^\perp-y_i^\perp\|+\|z^\#\| \\
&=& \|x^*-y_i^\perp\|=d(x^*,Y^\perp).
\eeqa
  Hence, $d(z^\perp,Y^\perp)=\|z^\perp-y_i^\perp\|$, $i=0,1,\dots,k$. Therefore, $y_i^\perp,y_0^\perp\in P_{Y^\perp}(z^\perp)$, $i=1,2,\dots,k$. As $z^\perp\in (X/Z)^*=Z^\perp$ and $y_i^\perp\in (Y/Z)^\perp=Y^\perp_{Z^\perp}=Y^\perp$,  $(Y/Z)$ does not have property-$(k-U)$ in $(X/Z)$.
\end{proof}

\begin{thm}\label{H6}
	Let $X$ be an $L_1$-predual (or an $M$-embedded) space, and let $Y$ be a finite co-dimensional subspace of $X$. Then, $Y$ has property-$(k-U)$ in $X$ if and only if $Y^{\perp\perp}$ has property-$(k-U)$ in $X^{**}$.
\end{thm}
\begin{proof}
Let $Y$ not have property-$(k-U)$ in $X$. Let $y^*\in S_{Y^*}$ and $x_1^*,\dots,x_{k+1}^*\in HB(y^*)$ be affinely independent. Then, the canonical images $\widehat{x_1^*},\dots,\widehat{x_{k+1}^*}\in HB(\widehat{y^*})$. It is known that $Y^{\perp\perp*}\cong Y^{***}$, $Y^{\perp\perp}$ does not have property-$(k-U)$ in $X^{**}$.

Suppose that $Y^{\perp\perp}$ does not have property-$(k-U)$ in $X^{**}$, that is, $Y^{\perp\perp\perp}$ is not a $k$-Chebyshev subspace of $X^{***}$. Because $Y$ is a finite co-dimensional subspace, we have $Y^{\perp}\cong Y^{\perp\perp\perp}$. It is known that $X^*$ is an $L$-summand in $X^{***}$ when $X$ is an $L_1$-predual or an $M$-embedded space. Thus, $Y^\perp$ is not a $k$-Chebyshev subspace of $X^{***}$, and $X^*$ is an $L$-summand in $X^{***}$. Using similar arguments in the proof of Theorem \ref{H5}, we find that $Y^\perp$ is not a $k$-Chebyshev subspace of $X^*$.  
\end{proof}

The next theorem establishes that property-$(k-U)$ remains hereditary for the class of subspaces that are ideals in a Banach space. Let us recall that if $Y$ is an ideal in $X$ and $P:X^*\ra X^*$ is a corresponding norm-$1$ projection, then $R(P)\cong Y^*$ (see \cite[Pg.297]{OP}).

\begin{thm}\label{P43}
 Let $Y$ be an ideal in $X$. If $X$ has property-$(k-U)$ in $X^{**}$, then $Y$ has property-$(k-U)$ in $Y^{**}$.
 \end{thm}

 \begin{proof}
 Let $Y$ not have property-$(k-U)$ in $Y^{**}$. Let $y^*\in S_{Y^*}$ and $y^{***}_1,\dots,y^{***}_{k+1}\in HB(y^*)$ be affinely independent. Now, $Y^*$ and $Y^{***}$ are subspaces (up to isometric) of $X^*$ and $X^{***}$, respectively. Thus, $X$ does not have property-$(k-U)$ in $X^{**}$.
 \end{proof}

\subsection{Property-$(k-U)$ and multismoothness in Banach spaces}

\bdfn\cite{LR}
Let $X$ be a Banach space and $x\in S_X$.
\bla
\item  $x$ is said to be a {\it $k$-smooth point} if $\dim (S(X^*,x))\leq k-1$.
\item A Banach space $X$ is said to be {\it $k$-smooth} if $\dim (S(X^*,x))\leq k-1$ for all $x\in S_{X}$.
\el
\edfn

\begin{thm}\label{421}
Let $Y$ be a subspace of $X$. Suppose every $y\in S_Y$ is a $k$-smooth point in $X$. Then, any reflexive subspace $Z$ of $Y$ has property-$(k-U)$ in $X$.
\end{thm}

\begin{proof}
Suppose every $y\in S_Y$ is a $k$-smooth point. Let $Z$ be a reflexive subspace of $X$ that does not have property-$(k-U)$ in $X$. Let $z^*\in S_{Z^*}$ and $x^*_1,\dots,x_{k+1}^*\in HB(y^*)$  be affinely independent. Because $Z$ is reflexive, there exists $z_0\in S_Z$ such that $z^*(z_0)=1$ and $x_i^*(z_0)=1$, $1\leq i\leq k+1$. Hence, $\dim (S(X^*,z_0))\geq k$, that is, $z_0$ is not a $k$-smooth point.
\end{proof}

In particular, we have the following.

\begin{prop}\label{P4}
Let $F$ be a finite-dimensional subspace of $X$. Suppose that every $x\in \emph{ext}(B_F)$ is a $k$-smooth point in $X$. Then, any subspace $Z$ of $F$ has property-$(k-U)$ in $X$.
\end{prop}

\begin{proof}
Let $x\in S_F$, then $(x_i)_{i=1}^{n}\ci \textrm{ext}(B_F)$ such that $x=\sum_{i=1}^{n}\lambda_ix_i$ for some $(\lambda_i)_{i=1}^n$ with $\la_i\geq 0$ and $\sum_i \la_i=1$. Clearly, we have $S(X^*,x)\ci \cap_{i=1}^{n}S(X^*,x_i)$. Hence, every $x\in S_F$ is a $k$-smooth point in $X$. The result follows from Theorem~\ref{421}.
\end{proof}

We now identify some subspaces of $c_0$ that have property-$(k-U)$. 

\begin{thm}\label{T15}
Let $F$ be a finite-dimensional subspace of $c_0$. Then, $F$ has property-$(k-U)$ in $c_0$ for some $k\in \mathbb{N}$.
\end{thm}

\begin{proof}
Because $F$ is a finite-dimensional subspace of $c_0$, we have that $\textrm{ext}(B_F)$ is finite set and $B_F=\textrm{conv}(\textrm{ext}(B_F))$. Let $\textrm{ext}(B_F)=\{x_1,\dots,x_n\}$. As $x_i\in c_0$ and $\|x_i\|_\iy=1$, there exists $k_i\in \mathbb{N}$ such that $|x_i(j)|=1$, where $j\in \{n_1,\dots, n_{k_i}\}$.
It follows that there exists a large $k$ such that all $x_i'$s are $k$-smooth. Proposition~\ref{P4} leads to the conclusion. 
\end{proof}

\begin{thm}
Let $F$ be a finite dimensional subspace of $c_0$, which is also an ideal in $c_0$, then $F^*$ is a $k$-strictly convex subspace of $\ell_1$ for some natural $k$.
\end{thm}

\begin{proof}
First, $F^*$ is a subspace of $\ell_1$. We have already observed that $F$ has property-$(k-U)$ for some $k$. It is proven that any point in $S_F$ is a $k$-smooth point. Hence, $F$ is actually a $k$-smooth subspace of $c_0$. Because $F =(F^*)^*$, we have that $F^*$ is k-strictly convex.
\end{proof}


\begin{thebibliography}{99}
	\bibitem[BR]{BR} S. Basu, T. S. S. R. K. Rao, {\em Some stability results for asymptotic norming properties of Banach spaces}, Colloq. Math. {\bf 75} (2) (1998) 271--284. 
	\bibitem[DPR]{DPR} S. Daptari, T. Paul, T. S. S. R. K. Rao, {\em Stability of unique Hahn-Banach extensions and associated linear projections}. To appear in linear multilinear algebra http://dx.doi.org/10.1080/03081087.2021.1945526.
	\bibitem[DU]{DU} J. Diestel, J. J. Uhl, {\em Vector measures}, Mathematical Surveys, {\bf 15}. American Mathematical Society, Providence, R.I. (1977).
	\bibitem[F]{F} S. R. Foguel, {\em On a theorem by A. E. Taylor}, Proc. Amer. Math. Soc. {\bf 9} (325), (1958).
	\bibitem[GYK]{GYK} G. Godefroy, M. Yahdi, R Kaufman, {\em The topological complexity of a natural class of norms of Banach spaces}, Proceedings of the International Conference "Analyse $\&$ Logique'' (Mons, 1997). Ann. Pure Appl. Logic, {\bf 111} (1-2) (2001), 3--13.
	
	\bibitem[HWW]{HWW} P. Harmand, D. Warner, W. Warner, {\em M-ideals in Banach spaces and Banach algebras}, Lecture Notes in Mathematics, {\bf 1547}, Springer-Verlag, Berlin, viii+387 (1993).
	\bibitem[H]{H} J. Hennefeld, {\em $M$-ideals, HB-subspaces, and compact operators}, Indiana Univ. Math. J. {\bf 28} (6) (1979) 927--934.
	\bibitem[LH]{LH} H. E. Lacey, {\em The isometric theory of classical Banach spaces}, Die Grundlehren der mathematischen Wissenschaften, {\bf 208}. Springer-Verlag, New York, Heidelberg (1974).
	\bibitem[L]{L} \AA. Lima, {\em Intersection properties of balls and subspaces in Banach spaces}, Trans. Amer. Math. Soc. {\bf 227} (1977) 1--62.
	\bibitem[LA]{LA} \AA. Lima, {\em Uniqueness of Hahn-Banach extensions and liftings of linear dependencies }, Math. Scand. {\bf 53} (1983) 97--113.
	\bibitem[LR]{LR} B. L. Lin, T. S. S. R. K. Rao, {\em Multismoothness in Banach spaces}, Int. J. Math. Math. Sci., Art. ID {\bf 52382} (2007) 12 .
	\bibitem[O]{O} $\grave{E}$. F. Oya, {\em Strong uniqueness of the extension of linear continuous functionals according to the Hahn-Banach theorem}, Math. Notes {\bf 43}, 1-2 (1988) 134--139.
	\bibitem[OP]{OP} E. Oja, M. Poldvere, {\em On subspaces of Banach spaces where every functional has unique norm-preserving extension}, Studia Math., {\bf 117} (3) (1996) 289--306.
	\bibitem[P]{P} R. R. Phelps, {\em Uniqueness of Hahn-Banach extensions and unique best approximation}, Trans. Amer. Math. Soc. {\bf 95} (1960) 238--255.
	\bibitem[RA]{RA} T. S. S. R. K. Rao, {\em On ideals in Banach spaces}, Rocky Mountain J. Math. {\bf 31} (2001) 595--609.
	\bibitem[R]{R} R. A. Ryan,  {\em Introduction to tensor products of Banach spaces}, Springer Monographs in Mathematics. Springer-Verlag London, Ltd., London, xiv+225 (2002).
	\bibitem[MO]{MO} D. Serre, {\em Why is the Hahn-Banach theorem so important?}, \url{https://mathoverflow.net/q/26568}
	\bibitem[S]{S} I. Singer, {\em Best approximation in normed linear spaces by elements of linear subspaces}, Die Grundlehren der mathematischen Wissenschaften, {\bf 171}, Springer-Verlag, New York-Berlin (1970).
	\bibitem[T]{T} A. E. Taylor, {\em The extension of linear functionals}, Duke Math. J. {\bf 5}, 538--547 (1939).
	\bibitem[X]{X} Ji Hong Xu, {\em Norm-preserving extensions and best approximations}, J. Math. Anal. Appl. {\bf 183} (3) (1994) 631--638.
\end{thebibliography}
\end{document}